\documentclass[12pt]{article}

\usepackage{amsmath}
\usepackage{amssymb}
\usepackage{amsthm}

\usepackage{amsthm}

\newtheorem{thm}{Theorem}
\newtheorem{prop}[thm]{Proposition}

\newtheorem{dfn}[thm]{Definition}
\newtheorem{lemm}[thm]{Lemma}

\def\be{\begin{equation}}
\def\ee{\end{equation}}

\def\R{\mathbb{R}}

\def\N{\mathbb{N}}
\def\Ri{\mathbb{R}\cup \{+\infty\}}

\def\pa{\partial }
\def\dom{\mathrm{dom}\,}

\def\lip{\mathrm{Lip}\,}
\def\l{\langle}
\def\r{\rangle}

\def\eps{\varepsilon}

\def\la{\lambda}

\def\ol{\overline}

\def\le{\leq}
\def\ge{\geq}
\def\be {\begin{equation}}
\def\ee {\end{equation}}
\def\ba {\begin{array}}
\def\ea {\end{array}}
\def\pa{\partial}

\def\dom{\mathrm{dom}\,}
\def\plr{\mathrm{plr}}

\title{Subdifferential determination of a primal lower regular function on a Banach space}
\author{M. Ivanov\thanks{Radiant Life Technologies Ltd., Nicosia, Cyprus, e-mail:milen@radiant-life-technologies.com}, M. Konstantinov\thanks{Faculty of Mathematics and Informatics, Sofia University,   5, James Bourchier Blvd, 1164 Sofia, Bulgaria, e-mail:matey@fmi.uni-sofia.bg},  N. Zlateva\thanks{Faculty of Mathematics and Informatics, Sofia University,   5, James Bourchier Blvd, 1164 Sofia, Bulgaria, e-mail:zlateva@fmi.uni-sofia.bg}}
\date{}

\begin{document}

\maketitle

\begin{abstract}
    We generalize to Banach space Thibault-Zagrodny Theorem that if $f$ and $g$ are primal lower regular functions and $\partial f = \partial g$ locally, then $f$ and $g$ locally differ by a constant.

    Our method consists in reduction to slopes considerations.
    \\[0.2cm]
\end{abstract}

\section{Introduction}\label{sec:intro}
		The primal lower regular functions are introduced  by R. Poliquin in
	\cite{poliq} under the name primal lower nice functions, for a modern account see \cite[Chapter 11]{thibault-book}. They are a useful abstraction of the framework for non-convex
	optimization proposed by R.T. Rockafellar, see e.g. \cite{rock-tams}, that is, the compositions of
	convex lower semicontinuous and twice smooth functions. The latter framework is
	really apt for the purpose, however doing calculus directly with those so called
	\emph{convexly composite} functions, e.g. \cite[11.1.2]{thibault-book} is cumbersome, because of having to account for two
	maps. In this respect primal lower regular functions come in handy.
	
	In \cite{poliq} R. Poliquin shows that primal lower regular functions share many properties with the convex
	functions, including a local version of Moreau-Rockafellar integrability theorem. Namely,
	if the subdifferentials of two primal lower regular functions coincide in a neighbourhood of some point,
	then in a neighbourhood of  that point the functions differ by a constant. In \cite{poliq} it was proved
	in finite dimensions and later in \cite{TZ} it was extended to Hilbert space setting, see also \cite[p.1113]{thibault-book}. These proofs, however, rely
	on the special properties of Moreau-Yousida regularization which allow transferring the subdifferential
	equality to the regularized functions, when the norm is Hilbertian. This does not work in a general Banach space.

As the title suggests,
	this work is devoted to extending this result to a general Banach space setting.
	
	The problem of extending the subdifferential determination of primal lower regular functions to Banach spaces is surprisingly tough.
	One of the reasons being that the domain of a primal lower regular function could potentially be quite irregular: only approximate convexity
	can be perhaps claimed \cite{mat} and even that is not fully established.
	
	The new method we employ here is based on \emph{slopes}, see e.g. \cite{iz-slopes} and the references therein, and
	it works, because it does not require linear (or convex) structure, so the irregularity of the domain is not such an issue.

	Let us recall the definition of a primal lower regular function.
	\begin{dfn}[cf. \cite{thibault-book}, Definition~11.1]
		Let $(E,\|\cdot\|)$ be a Banach space. Let the function $f:E\to\Ri$ be proper and lower semicontinuous. We say that $f$ is \emph{primal lower regular at} $\bar x\in E$ with coefficient $c>0$ and radius $\delta>0$, and write
		$$
			\bar x \in \plr_{c,\delta} f,
		$$
		if
		\begin{equation}
			\label{eq:plr-def}
			f(y) \ge f(x) + p(y-x) - c(1+\|p\|)\|y-x\|^2,
		\end{equation}
		for all $x,y\in B^\circ(\bar x;\delta)$, and all $p\in\partial f(x)$.
	
		The function $f$ is called \emph{primal lower regular} if
		$$
			\bigcup_{c,\delta>0} \plr_{c,\delta} f = E.
		$$
	\end{dfn}

We will now formulate our main result.
	\begin{thm}
		\label{thm:main}
		Let $(E,\|\cdot\|)$ be a Banach space. Let the functions $f,g:E\to\Ri$ be proper and lower semicontinuous. Let
		$$
			\bar x \in \plr _{c,\delta} f \cap \plr_{c,\delta}g \cap \dom f,
		$$
		for some $c,\delta > 0$. Let
		\begin{equation}
			\label{eq:subd=}
			\partial f(x) = \partial g(x),\quad \forall x \in B^\circ(\bar x, \delta).
		\end{equation}
		Then there is a constant $a\in\R$ such that
		\begin{equation}
			\label{eq:main-determ}
			f(x) = g(x) + a,\quad \forall x \in B^\circ(\bar x, \hat \delta),
		\end{equation}
		where
		\begin{equation}
			\label{eq:def-hat-del}
			\hat \delta :=\min\left\{\frac{\delta}{2}, \frac{1}{18c} \right\} .
		\end{equation}
	\end{thm}
	Here $\partial$ denotes the  Clarke-Rockafellar subdifferential, see Section~\ref{sec:clarke} for a brief review of the latter.

	This work is organized as follows. In Section~\ref{sec:slopes} we develop the slope techniques we need. In Section~\ref{sec:clarke} we recall what we use from the theory of Clarke-Rockafellar subdifferential and retrieve how subdifferential inclusions imply slopes inequalities. Finally, in Section~\ref{sec:plr} we prove the main result.

\section{Slope regularity}\label{sec:slopes}
Let $(X,d)$ be  a complete metric space. Functions like $g:X\to\Ri$ are assumed proper and lower semicontinuous.

For a function $g:X\to\R$ and a convex set $U\subset X$ define
$$
    \lip g (U) := \sup\left\{\frac{|g(y)-g(x)|}{d(y,x)}:\ x\neq y\in U\right\},
$$
and for $x\in X$
$$
    \lip g (x) := \inf_{\delta>0} \lip g(B(x;\delta)).
$$
Estimating the Lipschitz constant of the kind of functions will be using, is straightforward.
\begin{lemm}
    \label{lem:esitm-lip}
    Let $x_0\in X$, and
    $$
        g(x) := \varphi (d(x,x_0)),
    $$
    where $\varphi:\R^+\to\R^+$ is a convex, increasing smooth function such that $\varphi(0) {=} 0$. Then
    $$
        \lip g (x) \le \varphi'(d(x,x_0)),\quad\forall x\in X.
    $$
\end{lemm}
\begin{proof}
    Fix a $x\in X$ and a $\delta > d(x,x_0)$.

    We have that $\varphi'$ is positive and increasing. Therefore,
    $$
        \lip \varphi([0,\delta]) = \varphi'(\delta).
    $$
    So, if $x,y\in B(x_0;\delta)$, that is, $d(x,x_0),d(y,x_0)\in[0,\delta]$, then
    $$
        |g(y) - g(x)| \le \varphi'(\delta) |d(y,x_0)-d(x,x_0)| \le \varphi'(\delta) d(y,x),
    $$
    from the triangle inequality. So,
    $$
        \lip g(B(x_0;\delta)) \le \varphi'(\delta).
    $$
    But $\lip g(x) \le  \lip g(B(x_0;\delta))$ and $\displaystyle \inf_{\delta > d(x,x_0)}\varphi'(\delta) = \varphi'(d(x,x_0))$, because $\varphi$ is smooth.
\end{proof}

The \emph{local slope} of a function $f:X\to \Ri$ at $x\in\dom f$ is
$$
    |\nabla f|(x) = \limsup_{y\to x\atop y\neq x}\frac{[f(x)-f(y)]^+}{d(x,y)},
$$
where
$[t]^+ := \max\{0,t\} = (t+|t|)/2$. Therefore,  $|\nabla f|(x) = 0$ if $x$ is a local minimum to $f$, and if not
$$
    |\nabla f|(x) = \limsup_{y\to x\atop y\neq x}\frac{f(x)-f(y)}{d(x,y)}.
$$
We explicit that $|\nabla f|(x) = 0$ at the isolated points of $\dom f$, where the limit above is undefined.
This notion was introduced by De Giorgi, Marino and Tosques~\cite{GMT} and used thereafter for study of concepts as descent curves (e.g. \cite{DIL}),    error bounds (e.g. \cite{AC}),   subdifferential calculus and metric regularity (e.g. the monograph~\cite{I} and  references therein).
For determination of functions,   local slopes recently appear as an useful tool, see e.g. \cite{DLS,DMS,DS,iz-slopes},  as well as, in the convex case subgradients of minimal norm, see \cite{ASV}, and global slopes, see \cite{TZgs}.

We may consider $|\nabla f|$ as a function from $\dom f$ to $[0,\infty]$. Moreover,
\[
|\nabla (rf)|(x)=r|\nabla f|(x),\quad \forall r\ge 0, \ \forall x\in \dom f.
\]

In some sense, not a strict one, the local slope $|\nabla f|(x)$ can be regarded as an one-sided, or \emph{unilateral}, analogy of $\lip f(x)$. Formally, $|\nabla f|(x)$ is an unilaterisation of $\displaystyle \limsup_{y\to x}|f(x)-f(y)|/d(x,y)$. What we will be using is the following relationship.
\begin{lemm}
    \label{lem:nabla<lip}
    Let $f:X\to\R$ be proper and lower semicontinuous and let $g:X\to\R$ be locally Lipschitz. If $f+g$ has a local minimum at $x_0$ then
    \begin{equation}
        \label{eq:nabla<lip}
        |\nabla f|(x_0) \le \lip g(x_0).
    \end{equation}
\end{lemm}
\begin{proof}
    We can assume without loss of generality that $f(x_0)=g(x_0)=0$. So, there is $\varepsilon>0$ such that
    $$
        f(x) + g(x) \ge 0,\quad\forall x\in B(x_0;\varepsilon).
    $$
    Then for each $\delta\in(0,\varepsilon)$
    $$
        f(x) \ge -g(x) \ge -\lip g (B(x_0;\delta))d(x,x_0),\quad\forall x\in B(x_0;\delta).
    $$
    By definition
    $$
        |\nabla f|(x_0) \le \lip g (B(x_0;\delta)),
    $$
    and taking infimum over $\delta\in (0,\varepsilon)$, we get \eqref{eq:nabla<lip}.
\end{proof}

Later on we will be using the following simple observation.
\begin{lemm}
    \label{lem:nabla-restrict}
    Let $f$ be proper and lower semicontinuous and let $x_0\in\dom f$. Consider the sublevel set
    $$
        Y := \{x\in X:\ f(x) \le f(x_0)\},
    $$
    and the restriction of $f$ to it
    $$
        \ol f  := f \upharpoonright Y.
    $$
    Then
    \begin{equation}
        \label{eq:nabla-restrict}
        |\nabla f|(x) = |\nabla \ol f|(x),\quad\forall x\in Y.
    \end{equation}
\end{lemm}
\begin{proof}
    The definition of $|\nabla f|(x)$ can be rewritten as
    $$
        |\nabla f|(x) = \inf_{\varepsilon>0} \sup\left\{\frac{f(x)-f(y)}{d(x,y)}:\ y\in B(x;\varepsilon),\ f(y) \le f(x) \right\}.
    $$
    Obviously, if $x\in Y$, then $f(y) \le f(x)$ implies $y\in Y$ as well, and this gives~\eqref{eq:nabla-restrict}.
\end{proof}

Recall the famous Ekeland Variational Principle, eg.
\cite[p.45]{phelps}, or \cite[p.198]{thibault-book}.
\begin{thm}[Ekeland]
    \label{thm:eke}
    If $f:X\to\Ri$ is proper, lower semicontinuous, and bounded below, and $x_0\in\dom f$, then for each $\lambda>0$ there exists $x_\lambda$ such that
    \begin{equation*}
        \label{eq:eke-decr}
        \lambda d(x_\lambda,x_0) \le f(x_0) - f(x_\lambda),
    \end{equation*}
    \begin{equation*}
        \label{eq:eke-min}
        f(x) + \lambda d(x_\lambda,x) > f(x_\lambda), \quad\forall x\neq x_\lambda.
    \end{equation*}
\end{thm}
\begin{lemm}
    \label{lem:vp}
    Let $f:X\to\R$ be a proper and lower semicontinuous function, and let $g:X\to\R$ be locally Lipschitz. If $f+g$ is bounded below, then for each $\varepsilon > 0$ there exists $x_\varepsilon$ such that
    \begin{equation}
        \label{eq:eps-min}
        (f+g)(x_\varepsilon) \le \inf (f+g) + \varepsilon,
    \end{equation}
    and
    \begin{equation}
        \label{eq:vp}
        |\nabla f| (x_\varepsilon) \le \lip g (x_\varepsilon) + \varepsilon.
    \end{equation}
\end{lemm}
\begin{proof}
    Fix $\varepsilon > 0$ and pick a $x_0$ such that
    $$
        (f+g)(x_0) \le \inf (f+g) + \varepsilon.
    $$
    By Ekeland's Variational Principle  -- Theorem~\ref{thm:eke} -- applied for $(f+g)$ (instead of $f$), and $\eps$ (instead of $\la$), there is $x_\varepsilon$ which satisfies
    \begin{equation}
        \label{eq:eke-decr}
        \eps d(x_\eps,x_0) \le (f+g)(x_0) - (f+g)(x_\eps),
    \end{equation}
    \begin{equation}
        \label{eq:eke-min}
        (f+g)(x) + \eps d(x_\eps,x) > (f+g)(x_\eps), \quad\forall x\neq x_\eps.
    \end{equation}

    Now, \eqref{eq:eke-decr} implies $(f+g)(x_\varepsilon) \le (f+g)(x_0)$, and we get \eqref{eq:eps-min}.

    On the other hand, \eqref{eq:eke-min} means that the function
    $$
        x \to f(x) + (g(x) + \varepsilon d(x_\varepsilon,x))
    $$
    attains its minimum at $x_\varepsilon$. Obviously, $\lip (g(\cdot) + \varepsilon d(\cdot,x_\varepsilon))(x_\varepsilon) \le \lip g (x_\varepsilon) + \varepsilon$, so Lemma~\ref{lem:nabla<lip} gives \eqref{eq:vp}.
\end{proof}
\begin{dfn}
    \label{def:slope-pln}
    The function $f:X\to\Ri$ is \emph{regularly sloped of order~2 with coefficient} $c>0$ whenever for each $x\in \dom |\nabla f|$,
    \begin{equation}
        \label{eq:slope-pln}
        f(y) \ge f(x) - |\nabla f|(x)d(y,x) - c(1+|\nabla f|(x))d^2(y,x),\quad\forall y\in X .
    \end{equation}
\end{dfn}
\begin{thm}
    \label{thm:reg-slope-repr}
    Let $f:X\to\Ri$ be regularly sloped of order $2$ with coefficient $c>0$. Then for each $\bar x$ with
    $$
        0 < |\nabla f|(\bar x) < \infty
    $$
    there is a sequence $x_n\to \bar x$ such that
    \begin{equation}
        \label{eq:sr-slope}
        |\nabla f|(\bar x) = \lim_{n\to\infty}\frac{f(\bar x)-f(x_n)}{d(\bar x,x_n)},
    \end{equation}
    and
    \begin{equation}
        \label{eq:kapa-2-ord}
        |\nabla f|(x_n) - |\nabla f|(\bar x) < 2c(|\nabla f|(\bar x)+2)d(\bar x,x_n),\quad\forall n > 1/|\nabla f|(\bar x).
    \end{equation}
\end{thm}
\begin{proof}
    Fix $\bar x\in\dom |\nabla f|$ such that
    $
        r := |\nabla f|(\ol x) > 0$.
      To shorten the annotations, we can assume without loss of generality that
    $
        f(\bar x) = 0$.
    So, \eqref{eq:slope-pln} gives
    \begin{equation}
        \label{eq:f-splr}
        f(x) \ge - rd(x,\bar x) - c(r+1)d^2(x,\bar x), \quad\forall x\in X.
    \end{equation}

    Fix an arbitrary $n\in\N$ such that $n > 1/r$, thus
    \begin{equation}
        \label{eq:rn-def}
        r_n := r - 1/n > 0.
    \end{equation}
    Consider the function
    $$
        g_n(x) := r_n d(x,\bar x) + c(r+2)d^2(x,\bar x).
    $$
    Lemma~\ref{lem:esitm-lip} gives that
    \begin{equation}
        \label{eq:g-n-lip}
        \lip {g_n}(x) \le r_n + 2c(r+2)d(x,\bar x),\quad\forall x\in X.
    \end{equation}

    Note that the function $(f+g_n)$ is bounded below. Indeed, from \eqref{eq:f-splr} it follows that
    $$f(x) + g_n(x) \ge - \frac{d(x,\bar x)}{n} + cd^2(x,\bar x) \ge \min_{t\ge 0} \left(ct^2-\frac{t}{n}\right).$$

    It is not possible for the function $(f+g_n)$ to have a local minimum at $\bar x$, for otherwise Lemma~\ref{lem:nabla<lip} and \eqref{eq:g-n-lip} would give $r=|\nabla f|(\bar x) \le r_n<r$, contradiction. So, $(f+g_n)(\bar x) = 0 > \inf(f+g_n)$.

    Choose an $\varepsilon_n > 0$ such that
    \begin{equation}
        \label{eq:eps-n-1}
        \inf(f+g_n) + \varepsilon_n < 0,
    \end{equation}
    and
    \begin{equation}
        \label{eq:eps<1/n}
        \varepsilon_n < \frac{1}{n}.
    \end{equation}
    From Lemma~\ref{lem:vp} there is $x_n$ such that
    \begin{equation}
        \label{eq:eps-n-min}
        (f+g_n)(x_n) \le \inf (f+g_n) + \varepsilon_n,
    \end{equation}
    and
    \begin{equation}
        \label{eq:n-vp}
        |\nabla f| (x_n) \le \lip {g_n} (x_n) + \varepsilon_n.
    \end{equation}
    From \eqref{eq:eps-n-min}  and \eqref{eq:eps-n-1} it follows that $(f+g_n)(x_n) < 0$, so $f(x_n) < -g_n(x_n) \le -r_nd(x_n,\bar x)$. Thus,
    \begin{equation}
        \label{eq:x_n-slope}
        \frac{f(\bar x)-f(x_n)}{d(\bar x,x_n)} = - \frac{f(x_n)}{d(\bar x,x_n)} > r_n.
    \end{equation}
    Using again $(f+g_n)(x_n) < 0$ together with \eqref{eq:f-splr}, we estimate $0 > f(x_n) + g_n(x_n) \ge -d(x_n,\bar x)/n + cd^2(x_n,\bar x)$, so
    $$
        0 > d(x_n,\bar x)(cd(x_n,\bar x) - 1/n).
    $$
    Clearly, this implies
    \begin{equation}
        \label{eq:xn-close}
        d(x_n,\bar x) < \frac{1}{cn}.
    \end{equation}
    From \eqref{eq:n-vp} and \eqref{eq:g-n-lip} we have that $|\nabla f|(x_n) - |\nabla f| (\bar x) \le (r_n + 2c(r+2)d(x_n,\bar x) + \varepsilon_n) - r = (r_n-r+\varepsilon_n) +  2c(r+2)d(x_n,\bar x)$. From \eqref{eq:rn-def} and \eqref{eq:eps<1/n} it follows that $(r_n-r+\varepsilon_n) = \varepsilon_n-1/n < 0$. Thus,
    \begin{equation}
        \label{eq:3*}
        |\nabla f|(x_n) - |\nabla f|(\ol x) < 2c(r+2)d(x_n,\bar x).
    \end{equation}

    Having constructed the sequence $(x_n)_{n>1/r}$, we claim that it satisfies the conclusion of the theorem. Indeed, \eqref{eq:xn-close} means that $x_n\to\bar x$, as $n\to\infty$, and \eqref{eq:x_n-slope} plus $r_n\to r$, as $n\to\infty$, gives \eqref{eq:sr-slope}. Finally, \eqref{eq:3*} immediately implies~\eqref{eq:kapa-2-ord}.
\end{proof}

\begin{lemm}
    \label{lem:lsc-slope}
    If $f:X\to \Ri$ is lower semicontinuous and regularly sloped of order $2$ with coefficient $c>0$, then the function
    $$
        x \to |\nabla f|(x)
    $$
    is lower semicontinuous.
\end{lemm}
\begin{proof}
    Fix $x\in X$ and let $x_n\to x$ as $n\to\infty$. We have to show that
    $$
        |\nabla f|(x) \le \liminf_{n\to\infty} |\nabla f|(x_n).
    $$
    If $|\nabla f|(x_n)\to\infty$ this holds, so let $r\in\R$ be such that
    $$
        r > \liminf_{n\to\infty} |\nabla f|(x_n).
    $$
    There is a subsequence $x_{n_k}\to x$ as $k\to\infty$, such that $|\nabla f|(x_{n_k}) < r$ for all $k\in\mathbb{N}$. From \eqref{eq:slope-pln} it follows that
    $$
        f(y) \ge f(x_{n_k}) - rd(y,x_{n_k}) - c(1+r)d^2(y,x_{n_k})
    $$
    for any fixed $y\in X$ and all $k\in\mathbb{N}$. Passing to limit as $k\to\infty$ and taking into account that $f$ is lower semicontinuous, we get
    $$
        f(y) \ge f(x) - rd(y,x) - c(1+r)d^2(y,x),
    $$
    so, by definition $r\ge |\nabla f|(x)$.
\end{proof}
\begin{lemm}
    \label{lem:f-cont}
    If $f:X\to \Ri$ is lower semicontinuous and regularly sloped of order $2$ with coefficient $c>0$, and the sequence $x_n\to x$ is such that the sequence $(|\nabla f|(x_n))_{n\in\N}$ is bounded, then
    \begin{equation}
        \label{eq:f-cont}
        \lim_{n\to\infty} f(x_n) = f(x).
    \end{equation}
\end{lemm}
\begin{proof}
    Let $r > 0$ be such that $|\nabla f|(x_n) < r$ for all $n\in\N$. From \eqref{eq:slope-pln} it follows that
    $$
        f(x) \ge f(x_n) - rd(x,x_n) - c(r+1)d^2(x,x_n),\quad\forall n\in\N.
    $$
    Since $d(x,x_n) \to 0$, as $n\to\infty$, we have
    $$
        \limsup_{n\to\infty} f(x_n) \le f(x),
    $$
    and \eqref{eq:f-cont} follows from the lower semicontinuity of $f$.
\end{proof}

The following estimate is crucial for our approach. 
\begin{lemm}
    \label{lem:series}
    Let $a_n \ge 0$ and $b_n > 0$ be such that
    \begin{equation}
        \label{eq:an-conv}
        \sum_{n=0}^\infty a_n < \infty,
    \end{equation}
    and for some $c>0$
    \begin{equation}
        \label{eq:an-bn-cond}
        b_{n+1} - b_n < 2c (2 + b_n) a_n, \quad \forall n \ge 0.
    \end{equation}
    Then the sequence $(b_n)_0^\infty$ is bounded.
\end{lemm}
\begin{proof}
    Set
    $
        s := \sum_{n=0}^\infty a_n$.
    Let
    $
        A := \{0\}\cup\{n\in\N:\ b_n \le 1\}$.

    Fix an arbitrary $n\in\N\setminus A$. Let
    $
        k := \max\{i\in A:\ i < n \}$.
    This means that $b_k \le b := \max\{b_0,1\}$ and $b_i > 1$, for all $i$ such that $ k+1\le i\le  n$. So, we derive from \eqref{eq:an-bn-cond} that
    $$
        \frac{b_{i+1}}{b_i} < 1 + 2c \left(1 + \frac{2}{b_i}\right) a_i < 1 + 6c a_i,\quad\text{for all }i\text{ such that } k+1\le i\le  n-1.
    $$
    Since $\ln (1+t) \le t$ for $t\ge 0$, we have
    $$
        \ln b_{i+1} - \ln b_i < 6c a_i,\quad\text{for all }i\text{ such that } k+1\le i\le  n-1.
    $$
    Summing these, we get
    $$
        \ln b_n - \ln b_{k+1} < 6c \sum_{i=k+1}^{n-1} a_i < 6c s.
    $$
    Therefore,
    $$
        b_n < b_{k+1}e^{6cs}.
    $$
    But because $b_k \le b$, \eqref{eq:an-bn-cond} gives $b_{k+1} \le b + 2c(2+b)a_k < b + 2c(2+b)s$, so
    $$
        b_n < (b + 2c(2+b)s)e^{6cs},\quad\forall n\in N\text{ such that }\ b_n > 1.
    $$
\end{proof}

We will prove here a variant of Long Orbit Empty Value (LOEV) method, see \cite{iz-loev, det-prpr}. It concerns multivalued maps $S:X\rightrightarrows X$, where $(X,d)$ is a complete metric space. As usual, $
    \dom S := \{x\in X: \ S(x)\neq\varnothing\}$.

A $S$-\emph{orbit} is called any finite or infinite sequence $x_0,x_1,x_2\ldots$ such that $x_{i+1}\in S (x_i)$ for all $i$'s with a successor. We call $x_0$ \emph{the starting point} of the orbit. We differentiate between \emph{finite} and \emph{infinite} orbits. It stands to reason that for a finite orbit  $x_0,\ldots,x_n$, with $x_{i+1}\in S(x_i)$, for all $i$ such that $0\le i\le n-1$, the last point $x_n$ would be called \emph{the end point} of the orbit. Obviously, each finite orbit $(x_i)_0^n$ is of \emph{finite length}
$$
    \left|(x_i)_0^n\right| := \sum_{i=0}^{n-1} d(x_{i+1},x_i),
$$
if $n\ge 1$. The length of an orbit consisting of one point is, of course, zero.

Considering infinite orbits, $(x_i)_{0}^\infty$ with $x_{n+1}\in S(x_n)$ for all $n\ge0$ is called orbit of \emph{finite length} if
$$
    \left|(x_i)_0^\infty\right| :=  \sum_{n=0}^\infty d(x_{n+1},x_n) < \infty,
$$
and, respectively, of \emph{infinite length} if the above series is divergent.

Since $X$ is complete, each infinite orbit of finite length is convergent as a sequence, that is, there is $x_\infty\in X$ such that
$$
    \lim_{n\to\infty} d(x_n,x_\infty) = 0.
$$
We call $x_\infty$ the \emph{end point} of the orbit $(x_i)_{0}^\infty$.
\begin{dfn}
    \label{def:star}
    Let $(X,d)$ be a complete metric space and let $S:X\rightrightarrows X$ be a multivalued map. We say that $S$ satisfies property $(\ast)$ if
    \begin{equation}
        \label{eq:x-not-s}
        x\not\in S(x), \quad \forall x \in X,
    \end{equation}
    and for any infinite $S$-orbit $(x_n)_{0}^\infty$ of finite length with end point $x_\infty$, and any $y\in S(x_\infty)$, there exists $N\in\N$ such that
    $$
        y \in S(x_n), \quad \forall n > N.
    $$
\end{dfn}

\begin{thm}
    \label{thm:loev}
    Let $(X,d)$ be a complete metric space. Let $S:X\rightrightarrows X$ satisfy  property $(\ast)$ and let $x_0\in X$ be fixed.

    Then at least one of the following two holds:
    \begin{enumerate}
        \item[$(a)$]
            There is a $S$-orbit of finite length starting at $x_0$ and ending outside of $\dom S$.
        \item[$(b)$]
            There is a $S$-orbit of infinite length starting at $x_0$.
    \end{enumerate}
\end{thm}
\begin{proof}
    We construct the desired $S$-orbit by induction, starting at $x_0$.

    If $x_n$ is already chosen, then:

    If $S(x_n) = \varnothing$ then $(a)$ is fulfilled and we are done.

    If, on the other hand, $S(x_n) \neq \varnothing$, let
    \begin{equation}
        \label{eq:loev-nu}
        \nu_n := \sup\{d(y,x_n):\ y \in S(x_n)\} > 0.
    \end{equation}
    Take $x_{n+1} \in S(x_n)$ such that
    \begin{equation}
        \label{eq:loev-next}
        d(x_{n+1},x_n) > \min\{1, \nu_n/2\}.
    \end{equation}

    Assume that we obtain in this way an infinite orbit $(x_n)_{0}^\infty$ with $x_{n+1}\in S(x_n)$, $\forall n\ge0$.

    If $(x_n)_{0}^\infty$ is of infinite length then $(b)$ is fulfilled and we are done.

    If $(x_n)_{0}^\infty$ is of finite length, then let $x_\infty$ be the end point of $(x_n)_{0}^\infty$, i.e.
    $$
        \lim_{n\to\infty} x_n = x_\infty.
    $$
    Assume that there is some $y\in S(x_\infty)$. We know from $(\ast)$ property, see \eqref{eq:x-not-s}, that $y\neq x_\infty$, and $y\in S(x_n)$ for all but finitely many $n$'s. Then \eqref{eq:loev-nu} gives $\nu_n \ge d(y,x_n) \to d(y,x_\infty) > 0$ and \eqref{eq:loev-next} implies $\liminf_{n\to\infty} d(x_{n+1},x_n) > 0$, so $\left|(x_n)_{0}^\infty\right| = \infty$, contradiction.  Therefore, $S(x_\infty) = \varnothing$, and $(a)$ holds.
\end{proof}

\begin{lemm}[see \cite{iz-slopes}]
    \label{lem:gr-dens}
    Let $(X,d)$ be complete metric space. Let $f,g: X\to\Ri$ be lower semicontinuous. Let $f$ also be proper and bounded below. Then
    \begin{equation}
        \label{eq:gr-dens}
        \inf_{\dom f} (f-g) = \inf_{\dom |\nabla f|} (f-g).
    \end{equation}
\end{lemm}

In  \cite[Lemma~3.8.]{iz-slopes} the equality  \eqref{eq:gr-dens} is proved for the global slope, which dominates the local slope, so the formula there is stronger, but we do not need it in full.

\begin{dfn}
    \label{dfn:sharp-min}
    Let $(X,d)$ be complete metric space. Let $f: X\to\Ri$ be proper, lower semicontinuous and bounded below.

    We say that $\bar x$ is the \emph{sharp minimum} of $f$, if it is the minimum:
    $$
        f(\bar x) = \min f,
    $$
    and, moreover,
    \begin{equation}
        \label{eq:sharp}
        0 < \inf_{x\neq\bar x} |\nabla f|(x).
    \end{equation}
\end{dfn}
If $f$ has a sharp minimum, then $f$ cannot have any other local minima, because $|\nabla f|(x) = 0$ at each local minimum $x$. It is easy to check that the sharp minimum is \emph{strong}, that is, each minimising sequence converges to it, but a strong minimum is not necessarily sharp.
\begin{prop}
    \label{pro:slop-main}
    Let $(X,d)$ be a complete metric space and let $f,g:X\to\Ri$ be proper and lower semicontinuous functions. Let $f$ be regularly sloped of order $2$ with coefficient $c>0$ and have sharp minimum at $\bar x$. If
    \begin{equation}
        \label{eq:nabla-fg<}
        |\nabla g|(x) < |\nabla f|(x),\quad \forall x\in \dom |\nabla f|\setminus \{\bar x\},
    \end{equation}
    and
    \begin{equation}
        \label{eq:fg-meshano}
        g(y) \ge g(x) - |\nabla f|(x) d(y,x) - c( |\nabla f|(x) + 1)d^2(y,x),
    \end{equation}
    for all $x\in\dom|\nabla f|$ and all $y\in X$,
    then
    \begin{equation}
        \label{eq:nabla-fg<concl}
        \min_{\dom f} (f-g) = f(\bar x) - g(\bar x).
    \end{equation}
\end{prop}
\begin{proof}
    If $\bar x\not\in\dom g$ then the right hand side of \eqref{eq:nabla-fg<concl} is $-\infty$, this is the minimum, and \eqref{eq:nabla-fg<concl} is trivial.

    So, let $g(\bar x) < \infty$ and let us assume for simplicity -- and without loss of generality -- that
    $$
        f(\bar x) = g(\bar x) = 0.
    $$
    To verify \eqref{eq:nabla-fg<concl} we need to prove that $f\ge g$. But from Lemma~\ref{lem:gr-dens} we know that it is enough to show that
    $$
        f(x) \ge g(x),\quad \forall x\in \dom |\nabla f|.
    $$
    Therefore, fix an arbitrary $x_0\in X$ such that $f(x_0) < \infty$. In what follows we will show that $f(x_0) \ge g(x_0)$.

       Since $\bar x$ is a sharp minimum, from \eqref{eq:sharp} there is $\varepsilon > 0$ such that
    \begin{equation}
        \label{eq:eps-gap}
        |\nabla f|(x) > \varepsilon,\quad\forall x\in \dom|\nabla f|(x) \setminus \{\bar x\}.
    \end{equation}
    For any $x\in\dom|\nabla f|$ such that $x\neq \bar x$, define the following subsets of $\dom|\nabla f|$:
    \begin{equation}
        \label{eq:def-s1-diff}
        S_1(x) := \{y\in \dom|\nabla f|:\ f(y) - g(y) < f(x) - g(x)\}
    \end{equation}
  (note  that because of \eqref{eq:nabla-fg<}, $x\in\dom|\nabla g| \subset \dom g$, so $f(x) - g(x)$ is well defined);
    \begin{equation}
        \label{eq:def-s2-decr}
        S_2(x) := \{y\in \dom|\nabla f|:\ f(y) < f(x) - \varepsilon d(y,x)\};
    \end{equation}
    \begin{equation}
        \label{eq:def-s3-nabla}
        S_3(x) := \{y:\ |\nabla f|(y) < |\nabla f|(x) + 2c(2+|\nabla f|(x))d(y,x)\},
    \end{equation}
    and
    $$
        S(x) := S_1(x) \cap S_2(x) \cap S_3 (x).
    $$
    In this way we define a multivalued map $S:X\rightrightarrows X$. We will study its properties.

    First, the way it is defined ensures that $\dom S \subset \dom|\nabla f| \setminus\{\bar x\}$. We claim that in fact
    \begin{equation}
        \label{eq:dom-S}
        \dom S = \dom|\nabla f| \setminus\{\bar x\}.
    \end{equation}
    Indeed, fix $x\in \dom|\nabla f|$ such that $x\neq \bar x$, so
    $$
        r := |\nabla f|(x) \in (\varepsilon,\infty),
    $$
    see \eqref{eq:eps-gap}. From \eqref{eq:nabla-fg<} we have $|\nabla g|(x) < r$, so there is $\delta > 0$ such that $r-\delta > 0$ and
    $$
        g(y) \ge g(x) - (r-\delta) d(y,x),\quad\forall y\in B(x;\delta).
    $$
    From Theorem~\ref{thm:reg-slope-repr} there is a sequence $x_n\to x$ satisfying \eqref{eq:sr-slope} and \eqref{eq:kapa-2-ord}. Now, \eqref{eq:kapa-2-ord} simply means that $x_n\in S_3(x)$ for all large $n$'s. From \eqref{eq:sr-slope} we get on the one hand that, $f(x) - f(x_n) > \varepsilon d(x,x_n)$, because $|\nabla f|(x) > \varepsilon$. Therefore, $x_n\in S_2(x)$ for all large $n$'s. On the other hand, $x_n\in B(x;\delta)$ for all large $n$'s, so using the above inequality for $g$ and again \eqref{eq:sr-slope} we get
    $$
        \frac{f(x)-f(x_n)}{d(x,x_n)} > r-\delta \ge \frac{g(x)-g(x_n)}{d(x,x_n)} \Rightarrow f(x_n) - g(x_n) < f(x) - g(x),
    $$
    for all $n$ large enough. This means that $x_n\in S_1(x)$ for all large $n$'s. So, $x_n\in S(x)$ for all but finitely many $n$'s and, therefore, $S(x)\neq\varnothing$ and \eqref{eq:dom-S} is verified.

    Next, each $S$-orbit starting at $x_0$ is of finite length. Indeed, if $(x_n)_{0}^\infty$ is a $S$-orbit, it is also a $S_2$-orbit, so iterating on \eqref{eq:def-s2-decr} we get
    $$
        \sum_{i=0}^n d(x_{i+1},x_i) < \varepsilon^{-1} \sum_{i=0}^n (f(x_i) - f(x_{i+1})) = \varepsilon^{-1} (f(x_0) - f(x_{n+1})),
    $$
    and, since $f\ge 0$, we get
    \begin{equation}
        \label{eq:finite-length}
        |(x_n)_{0}^\infty| \le f(x_0)/\varepsilon.
    \end{equation}
    The same estimate is valid, of course, for the lengths of the finite $S$-orbits starting at $x_0$ (but we do not need it).

    Let $(x_n)_{0}^\infty$ be an arbitrary infinite $S$-orbit starting at $x_0$. We  already know that it is of finite length, so it has an end point, say $x_\infty$. We claim that
    \begin{equation}
        \label{eq:cont-on-orbit}
        \lim_{n\to\infty} (f(x_n),g(x_n)) = (f(x_\infty),g(x_\infty)).
    \end{equation}
    Indeed, set $a_n := d(x_{n+1},x_n)$ and $b_n := |\nabla f|(x_n)$. We know that $a_n$'s satisfy \eqref{eq:an-conv} and because $(x_n)_{0}^\infty$ is a $S_3$-orbit, \eqref{eq:def-s3-nabla} translates to \eqref{eq:an-bn-cond}. From Lemma~\ref{lem:series} we get that the sequence $(|\nabla f|(x_n))_{0}^\infty$ is bounded and then Lemma~\ref{lem:f-cont} ensures that $f(x_n)\to f(x_\infty)$. We repeat the proof of  Lemma~\ref{lem:f-cont} for $g$ using \eqref{eq:fg-meshano}. Thus \eqref{eq:cont-on-orbit} is verified.

    Now we can show that $S$ has  property $(\ast)$. Because of the strict inequality, $S_1$ satisfies \eqref{eq:x-not-s},  thus so does $S$.

    Let $(x_n)_{0}^\infty$ be an infinite $S$-orbit of finite length ending at $x_\infty$ and let $y\in S(x_\infty)$. Because $y\in S_1(x_\infty)$, we have by \eqref{eq:def-s1-diff} that $f(y) - g(y) < f(x_\infty) - g(x_\infty)$, then \eqref{eq:cont-on-orbit} gives that $f(y) - g(y) < f(x_n) - g(x_n)$ for all large $n$'s. That is, eventually $y\in S_1(x_n)$.

    In a similar fashion, since $y\in S_2(x_\infty)$, we have $f(y) < f(x_\infty) - \varepsilon d(y,x_\infty)$, see \eqref{eq:def-s2-decr}. Since the function
    $x \to f(x) - f(y) - \varepsilon d(y,x)$ is lower semicontinuous and with positive value at $x_\infty$, we have $f(x_n) - f(y) - \varepsilon d(y,x_n) > 0$ for all but finitely many $n$'s, ergo $y\in S_2(x_n)$ eventually.

    Since $y\in S_3(x_\infty)$, we have by \eqref{eq:def-s3-nabla} that
    $$
        |\nabla f|(y) < |\nabla f|(x_\infty) + 2c(2 + |\nabla f|(x_\infty)) d(y,x_\infty).
    $$
    From Lemma~\ref{lem:lsc-slope} we have that the function $x\to |\nabla f|(x)$ is lower semicontinuous, thus the function $x\to |\nabla f|(x) + (2c + |\nabla f|(x)) d(y,x)$ is lower semicontinuous, so if we replace $x_\infty$ by $x_n$ in the above inequality, it will still be valid for large $n$'s, meaning that $y\in S_3(x)$ eventually.

    Property $(\ast)$ of $S$ being established, we can now apply Theorem~\ref{thm:loev}. From \eqref{eq:finite-length} it follows that for $S$ only (a) is possible.

    If a finite $S$-orbit starting at $x_0$, say $x_0,x_1,\ldots,x_n$, ends outside of $\dom S$, then because $S$ maps $\dom |\nabla f| \setminus \{\bar x\}$ into $\dom |\nabla f|$, necessarily $x_n\in \dom |\nabla f|$. Since $S(x_n) = \varnothing$, from \eqref{eq:dom-S} it follows that $x_n=\bar x$. Iterating on \eqref{eq:def-s1-diff}, we get $f(x_0) - g(x_0) > 0$.

    Finally, let $(x_n)_{0}^\infty$ be an infinite $S$-orbit of finite length starting at $x_0$ and ending at $x_\infty$ with $S(x_\infty) = \varnothing$. In the same way as in the proof of \eqref{eq:cont-on-orbit} we see that $(|\nabla f|(x_n))_{0}^\infty$ is a bounded sequence, and then Lemma~\ref{lem:lsc-slope} gives that $x_\infty\in \dom|\nabla f|$. Since $S(x_\infty) = \emptyset$, we have $x_\infty = \bar x$. Since $(x_n)_{0}^\infty$ is a $S_1$-orbit, we have $f(x_n) - g(x_n) < f(x_0) - g(x_0)$ for all $n\ge0$ and \eqref{eq:cont-on-orbit} yields $f(x_0) \ge g(x_0)$. (We can actually obtain strict inequality, but this is not important.)
\end{proof}
\begin{thm}
    \label{thm:slope-main}
    Let $(X,d)$ be a complete metric space and let $f,g:X\to\Ri$ be proper and lower semicontinuous functions. Let $f$ be regularly sloped of order $2$ with coefficient $c>0$ and have sharp minimum at $\bar x$. If
    \begin{equation}
        \label{eg:nab-fg<=}
        |\nabla g|(x) \le |\nabla f|(x),\quad\forall x\in X,
    \end{equation}
    and \eqref{eq:fg-meshano} is satisfied, then
    \begin{equation}
        \label{eg:slope-main-concl}
        f(x) \ge (f(\bar x) - g(\bar x)) + g(x),\quad\forall x\in X.
    \end{equation}
\end{thm}
\begin{proof}
    Since $\bar x$ is the minimum of $f$, we have $|\nabla f|(\bar x) = 0$. From \eqref{eg:nab-fg<=} it follows that $|\nabla g|(\bar x) = 0$ and in particular $g(\bar x) < \infty$. This means that the right hand side of \eqref{eg:slope-main-concl} is well defined, and also we can assume without loss of generality that
    $$
        f(\bar x) = g(\bar x) = 0.
    $$
    We have to prove that
    \begin{equation}
        \label{eq:4:f-ge-g}
        f(x) \ge g(x),\quad \forall x\in X.
    \end{equation}
    Since $f \ge 0$,
    $$
        f = \inf_{\delta > 0} (1 + \delta) f.
    $$
  For $\delta >0$ set $f_\delta := (1+\delta)f$. Since \eqref{eq:4:f-ge-g} is trivial if $f(x) = \infty$, it is sufficient to prove that
    $$
        f_\delta (x) \ge g(x),\quad\forall x\in\dom f,\ \forall \delta > 0.
    $$
    Fix an arbitrary $\delta > 0$. By definition,
    $$
        |\nabla f_\delta| (x) = (1+\delta) |\nabla f|(x),\quad\forall x\in X.
    $$
    From this and \eqref{eg:nab-fg<=} it follows that
    $$
    |\nabla f_\delta| (x) > |\nabla g|(x),\quad\forall x\in \dom |\nabla f_\delta|\setminus \{\bar x\},
    $$
    that is -- because $\dom|\nabla f_\delta| = \dom |\nabla f|$ -- the functions $f_\delta$ and $g$ satisfy \eqref{eq:nabla-fg<}. Clearly, $\bar x$ is sharp minimum to $f_\delta$, so Proposition~\ref{pro:slop-main} gives
    $$
        f_\delta(x) - g(x) \ge 0,\quad\forall x\in \dom f,
    $$
    see~\eqref{eq:nabla-fg<concl}.
\end{proof}

\section{Clarke-Rockafellar subdifferential}
    \label{sec:clarke}

To the end of the article we work in a Banach space denoted by $(E,\|\cdot\|)$ with closed unit ball $B_E$ and dual $E^*$.
	
    Provided $f:E\to \Ri$ is lower semicontinuous at $x\in\dom f$, the Clarke-Rockafellar directional derivative,
	e.g. \cite[p.97]{Clarke}, at the point $x$ at direction $h\in E$ is:
	\begin{equation}
		\label{eq:^o-def}
		f^\circ (x;h) := \sup_{\varepsilon > 0} f^\circ_\varepsilon (x;h),\text{ where }
	\end{equation}
	$$
	f^\circ_\varepsilon (x;h) := \limsup_{y\to_f x,t\searrow0} \frac{\inf f(y+t(h+\varepsilon B_X))-f(y)}{t}.
	$$
	The Clarke-Rockafellar subdifferential of $f$ at $x\in\dom f$ is defined as:
	\begin{equation}
		\label{eq:cr-subd-def}
		\partial f(x) := \{p\in E^*:\ p(h)\le f^\circ(x;h),\ \forall h\in E\}.
	\end{equation}
	Of course, $\partial f(x) = \varnothing$ if $f(x)=+\infty$.
    It can be shown that $\dom\partial f = \{x\in \dom f: \ f^\circ(x;h) > - \infty,\ \forall h \in E\} = \{x\in \dom f: \ \exists c > 0:\ f^\circ(x;h) \ge - c\|h\|,\ \forall h \in E\}$, and that on the latter set the function $h\to f^\circ (x;h)$ is sublinear, thus convex, and lower semicontinuous function from $E$ to $(-\infty,\infty]$. It is clear that $\partial f(x)$ is convex and $w^*$ closed subset of $E^*$.

	Note that if $f$ is $\lambda$-Lipschitz around $x$, then $\inf f(y+t(h+\varepsilon B_E)) \ge f(y+th) - \lambda t\varepsilon$, so for a Lipschitz function  \eqref{eq:^o-def} reduces to the much simpler Clarke's formula:
	$$
	    f^\circ (x;h) = \limsup_{y\to_f x,t\searrow0} \frac{ f(y+th)-f(y)}{t}.
	$$


     We consider also the Fr\'echet subdifferential of $f$ at $x\in \dom f$:
    $$
        p\in\partial_F f(x) \iff \liminf_{\|h\|\to0}\frac{f(x+h)-f(x)-p(h)}{\|h\|} \ge 0,
    $$
    otherwise $\pa _F(x)=\varnothing$.
    It is easy to check that $\partial_F\subset\partial$, see e.g. \cite[Proposition 4.8(a)]{thibault-book}.

    The function $f:E\to \Ri$ is called \emph{$F$-regular} at $x\in \dom f$ if
    $$
    \partial_F f(x) = \partial f(x),
        $$
   and it is called  \emph{$F$-regular} if it is so at any point $x\in \dom f$.

 The basic example of   $F$-regular functions are the convex continuous functions.

    We will use the following Sum Theorem.
    \begin{thm}[Sum Theorem]
        \label{thm:sum}
		Let $f:E\to\Ri$ be proper,  lower semicontinuous and $F$-regular function. Let $h:E\to\R$ be locally Lipschitz and $F$-regular function. Then $f+h$ is $F$-regular, and
		\begin{equation}
			\label{eq:sum}
			\partial(f+h) (x) = \partial f (x) + \partial h (x),\quad \forall x\in E.
		\end{equation}
	\end{thm}
\begin{proof}
First, let $x\in \dom   f$. From \cite[Theorem 2.98]{thibault-book} we have that
\[
\pa (f+h)(x)\subset \pa f(x)+\pa h(x).
\]
To get the opposite inclusion, we use that
\[
\pa f(x)+\pa h(x)=\pa_F f(x)+\pa_F h(x)
\]
by $F$-regularity of both functions, and that
\[
\pa _F f(x)+\pa _F h(x)\subset \pa _F (f+h)(x),
\]
see e.g. \cite[Proposition 4.11]{thibault-book}.
Hence,
\[
\pa f(x)+\pa h(x)\subset \pa _F (f+h)(x)\subset   \pa   (f+h)(x), \quad \forall x\in \dom f.
\]
If $x\not\in \dom f$, then $x\not \in \dom (f+h)$ either, hence $\pa (f+h)(x)=\pa f(x)=\varnothing$. Then $\pa  f(x)+\pa h (x)=\varnothing$ by definition, and \eqref{eq:sum} is established.
\end{proof}

    $F$-regularity provides a concise condition allowing the transfer of subdifferential inclusions to slope inequalities.
    \begin{prop}
        \label{pro:subdf-to-slope}
       Let $f:E\to\Ri$  be proper, lower semicontinuous and $F$-regular. Then
       \begin{equation}
           \label{eq:dom-coinc}
           \dom|\nabla f| = \dom\partial f ,
       \end{equation}
       and
       \begin{equation}
           \label{eq:slope-charact}
           |\nabla f|(x) = \min \{\|p\|:\ p\in\partial f(x)\},\quad\forall x\in\dom|\nabla f| .
       \end{equation}
    \end{prop}
    \begin{proof}
        Note that, because the intersection of $\partial f(x)$ with a dual ball is $w^*$-compact, and the dual norm is $w^*$ upper semicontinuous, the minimum in the right hand side of \eqref{eq:slope-charact} is really attained.

        Let $x_0\in\dom|\nabla f|$ and let $r>|\nabla f|(x_0)$ be arbitrary. By definition, the function
        $$
          u(x) := f(x) + r\|x-x_0\|
        $$
        attains a local minimum at $x_0$, so $0\in\partial u(x_0)$. By Sum Theorem -- Theorem~\ref{thm:sum} -- there are $p\in\partial f(x_0)$ and $q\in \pa r\|\cdot -x_0\|(x_0)$ such that $\|q\|\le r$, and $p+q=0$. Clearly,  $\|p\|\le r$. Thus we get
        \begin{equation}
            \label{eq:dom-nab-in-pa}
            \dom|\nabla f| \subset \dom\partial f,
        \end{equation}
        and, taking infimum over $r > |\nabla f|(x_0)$ we obtain that
        \begin{equation}
            \label{eq:nab-ge-pa}
            |\nabla f| (x_0) \ge \|p\| \ge \min \{\|p\|:\ p\in\partial f(x_0)\}.
        \end{equation}
        Note that for these we did not use the $F$-regularity of $f$.

        Let now $x_0\in\dom\partial f$ and let $p_0\in\partial f(x_0)$ be such that
        $$
            \|p_0\| = \min \{\|p\|:\ p\in\partial f(x_0)\}.
        $$
        Since $f$ is $F$-regular, $p_0\in\partial_Ff(x_0)$, that is,
        $$
            f(x) - f(x_0) \ge p_0(x-x_0) + \alpha(x)\|x-x_0\|,
        $$
        where $\alpha(x) \to 0$, as $x\to x_0$. So, $f(x) - f(x_0) \ge - (\|p_0\|- \alpha(x))\|x-x_0\|$, and
        $$
            \liminf_{x\to x_0} \frac{f(x)-f(x_0)}{\|x-x_0\|} \ge -\|p_0\|,
        $$
        meaning that
        \begin{equation}
            \label{eq:nab-le-p0}
            |\nabla f|(x_0) \le \|p_0\|.
        \end{equation}
        So, $x_0\in\dom|\nabla f|$ and, since $x_0\in\dom\partial f$ was arbitrary, we have $\dom\partial f \subset \dom|\nabla f|$, which together with \eqref{eq:dom-nab-in-pa} gives \eqref{eq:dom-coinc}. On the other hand, \eqref{eq:nab-ge-pa} and \eqref{eq:nab-le-p0} give \eqref{eq:slope-charact}.
    \end{proof}

    \section{Primal lower regular functions}
    \label{sec:plr}
We start with some preliminary results concerning primal lower regular functions defined on a Banach space $(E,\|\cdot\|)$.

First, observe that if  $\bar x \in\plr_{c,\delta}f$, then $f$ is $F$-regular on  $  B^\circ(\bar x;\delta)$, see \cite[Thorem 11.16(a)]{thibault-book}.
\begin{lemm}
    \label{lem:af-plr}
    If $\bar x \in\plr_{c,\delta}f$ for some $c,\delta>0$, and $\alpha\in(0,1)$, then
    $$
        \bar x \in \plr_{c,\delta}(\alpha f).
    $$
\end{lemm}
\begin{proof}
  Let $x,y\in B^\circ(\bar x;\delta)$ and $p\in\partial f(x)$ be arbitrary. Multiplying \eqref{eq:plr-def} by $\alpha > 0$, and using that $p_1 := \alpha p \in \partial (\alpha f)(x)$, we get
    \begin{eqnarray*}
        \alpha f(y) &\ge& \alpha f(x) + p_1(y-x) -c(\|p_1\| + \alpha) \|y-x\|^2\\
        &\ge& \alpha f(x) + p_1(y-x) -c(\|p_1\| + 1) \|y-x\|^2.
    \end{eqnarray*}
    Because $\partial (\alpha f) = \alpha \partial f$, we are done.
\end{proof}
\begin{lemm}
    \label{lem:pln+conv}
    Let $f:E\to\Ri$ be a proper and lower semicontinuous function, and let $h:E\to\R$ be a convex and continuous function.

    If
    $$
        \bar x \in \plr_{c,\delta} f \ \ \text{ and }\ \  \lip h(B^\circ(\bar x;\delta)) = L,
    $$
    then
    $$
        \bar x \in \plr_{c(L+1),\delta} (f + h).
    $$
\end{lemm}
\begin{proof}
    Denote $u:=f+h$ and fix arbitrary $x\in B^\circ(\bar x;\delta)$ and $\xi\in\partial u(x)$.

    From the Sum Theorem -- Theorem~\ref{thm:sum} -- we have that $\xi = p + q$ where $p\in\partial f(x)$ and $q\in\partial h(x)$. The latter means
    $$
        h(y)\ge h(x) + \l q,y - x\r,\;\forall y\in X.
    $$
    Adding this and \eqref{eq:plr-def} together yields
    \be\label{eq:pln+convex_h_eq}
        u(y)\ge u(x) + \l \xi,y - x\r - c\left(\|p\|+1\right)\|y - x\|^2,\quad\forall y\in B^\circ(\bar x;\delta).
    \ee
    Note that, since $\|p\| = \|\xi-q\| \le \|\xi\|+\|q\|$,
    \[
        \|p\|+1 \le \|\xi\|+\|q\|+1 \le (\|q\| +1)(\|\xi\|+1),
    \]
    and because $\|q\|\le L$, the Lipshitz constant of $h$ around $\bar x$, \eqref{eq:pln+convex_h_eq} gives
    \[
        u(y) \ge u(x) + \l\xi,y-x\r - c(L +1)(\|\xi\|+1)\|y - x\|^2,\quad\forall y\in B^\circ(\bar x;\delta).
    \]
\end{proof}

\begin{lemm}
    \label{lem:plr-to-sharp}
    Let $f:X\to\Ri$ be a proper and lower semicontinuous function. Let
    \begin{equation}
        \label{eq:ps-x-plr}
        \bar x \in \plr_{c,\delta} f,
    \end{equation}
    for some $c , \delta > 0$,
    and let $p\in\partial f(\bar x)$ be such that
    \begin{equation}
        \label{eq:ps-p<1}
        \| p \| < 1.
    \end{equation}
    For the function
    \begin{equation}
        \label{eq:ps-f1-def}
        f_1(x) := f(x) - p(x) + 4 \|x-\bar x\|
    \end{equation}
    it is fulfilled that
    $$
        \bar x \in \plr_{c',\delta'} f_1,
    $$
    with
    \begin{equation}
        \label{eq:ps-cdel-prim}
        c' = 6c,\quad \delta' = \min\left\{\delta,\frac{1}{9c}\right\}.
    \end{equation}
   Moreover, $f_1(\bar x) \le \inf f_1(B^\circ(\bar x;\delta'))$, and
    \begin{equation}
        \label{eq:ps-sharp-min}
        |\nabla f_1|(x) \ge 1,\quad\forall x\in B^\circ(\bar x;\delta') \setminus \{\bar x\}.
    \end{equation}

    \begin{proof}
        Assume that, contrary to \eqref{eq:ps-sharp-min}, there is $x\in B^\circ(\bar x;\delta') \setminus \{\bar x\}$ such that $|\nabla f_1|(x) < 1$. From Proposition~\ref{pro:subdf-to-slope}, see \eqref{eq:slope-charact}, there is $\xi\in \partial f_1(x)$ such that $\|\xi\| < 1$. By Sum Theorem -- Theorem~\ref{thm:sum} -- we have that
        $$
            \xi = q - p + r,
        $$
        where $q\in\partial f(x)$ and $r\in\partial (4\|\cdot-\bar x\|)(x)$. It is immediate that, since $x\neq\bar x$,
        $$
            \|r\| = 4,\text{ and }r(\bar x - x) = -4\|\bar x - x\|.
        $$
        Since $\|\xi\|<1$ and $\|p\| < 1$, we have
        $$
            q(\bar x - x) = (\xi+p)(\bar x - x) - r (\bar x -x) \ge 2\|\bar x - x\|.
        $$
        By definition and using that $\|q\| \le 2 +\|r\| = 6$,
        \begin{eqnarray*}
            f(\bar x) &\ge& f(x) + q(\bar x - x) - c(\|q\| + 1)\|\bar x - x\|^2\\
            &\ge& f(x) + 2\|\bar x - x\| - 7 c \|\bar x - x\|^2.
        \end{eqnarray*}
        But, again by definition, and using that $\|p\| < 1$,
        \begin{eqnarray*}
            f(x) &\ge& f(\bar x) + p(x - \bar x) - c(\|p\| + 1)\|x - \bar x\|^2\\
            &>& f(\bar x) - \|x - \bar x\| - 2 c \|x - \bar x\|^2.
        \end{eqnarray*}

         Now, putting the above two together, we get
        $$
            0 > \|x - \bar x\| - 9c \|x - \bar x\| ^2,
        $$
        meaning that $\|x - \bar x\| > 1/9c \ge \delta'$, see \eqref{eq:ps-cdel-prim}, contradiction.

        Since the function
        $$
         h(x)=   -p(x) + 4 \|x-\bar x\|
        $$
        is convex and globally Lipschitz with constant less or equal to $ \|p\| + 4 < 5$, Lemma~\ref{lem:pln+conv} immediately gives that $\bar x \in \plr_{c',\delta} f_1$ with $c'$ given by \eqref{eq:ps-cdel-prim}. Since $\delta' \le \delta$, we are done.

        Moreover,        for any $x\in B^\circ(\bar x;\delta')$ it holds that
        \[
       f(x) \ge f(\bar x) + p(x - \bar x) - c(\|p\| + 1)\|x - \bar x\|^2.
        \]
      Hence, having in mind that $\|p\|<1$,
  \[
       f(x)-p(x) \ge f(\bar x) - p(\bar x) - 2c \|x - \bar x\|^2,
        \]
  \[
       f(x)-p(x) +4\| x-\bar x\| \ge f(\bar x) - p(\bar x) +4\| x-\bar x\|- 2c \|x - \bar x\|^2,
        \]
  \[
       f_1(x) \ge f_1(\bar x)  +2\| x-\bar x\|(2- c \|x - \bar x\|).
        \]

  Since  $2c \|x - \bar x\| < 1$ for $x\in B^\circ(\bar x;\delta')$, we have per force that $f_1(\bar x) \le \inf f_1(B^\circ(\bar x;\delta'))$.
    \end{proof}

    \begin{lemm}
        \label{lem:determ}
        Let the functions $f,g:E\to\Ri$ be proper and lower semicontinuous, and let for some $c,\delta > 0$,
		\begin{equation}
			\label{eq:x-bar-in-plr-fg-and-dom}
			\bar x \in \plr _{c,\delta} f \cap \plr_{c,\delta}g\cap\dom\partial f.
		\end{equation}
       Let
        \begin{equation}
            \label{eq:subd-incl}
            \partial f(x) \subset \partial g(x),\quad \forall x \in B^\circ(\bar x; \delta).
        \end{equation}
        Then
        \begin{equation}
			\label{eq:main-determ}
			f(x) \ge g(x) + (f(\bar x) - g(\bar x)),\quad \forall x \in B^\circ(\bar x;  \delta'),
		\end{equation}
        where $\delta'$ is given by \eqref{eq:ps-cdel-prim}.
    \end{lemm}
    \begin{proof}
        First, let us note that from \eqref{eq:x-bar-in-plr-fg-and-dom} and \eqref{eq:subd-incl} it follows that $\bar x \in \dom g$, so the right hand side of \eqref{eq:main-determ} is well defined.

        We can assume without loss of generality that $\bar x = 0$ and $f(\bar x) = g(\bar x) = 0$.

        Fix an arbitrary $x_0\in \delta'B^\circ_E$. We need to show that $f(x_0) \ge g(x_0)$, so we can assume without loss of generality that $f(x_0) < \infty$.

        Let $\nu\in(0,1)$ be such that
        $$
            \nu + \|x_0\| < \delta'.
        $$
        From Proposition~\ref{pro:subdf-to-slope} it follows that $0\in \dom |\nabla f|$. Pick $\alpha \in (0,1)$ such that
        \begin{equation}
            \label{eq:alpha-choice}
            |\nabla (\alpha f)|(0) < \min\{1,\nu/\delta '\},\ \inf (\alpha f)(\delta'B_E) > -\nu,\ (\alpha f)(x_0) < \nu.
        \end{equation}
        Note that by definition $f$ is bounded below on $\delta' B_E$, so such choice is possible. From Lemma~\ref{lem:af-plr} and \eqref{eq:x-bar-in-plr-fg-and-dom} we have that
        \begin{equation}
            \label{eq:plr-alfa-fg}
            0 \in \plr _{c,\delta} (\alpha f) \cap \plr _{c,\delta} (\alpha g).
        \end{equation}
        From \eqref{eq:alpha-choice} and Proposition~\ref{pro:subdf-to-slope}, see \eqref{eq:slope-charact}, it follows that there is a $p\in\partial (\alpha f)(0)$ such that
        \begin{equation}
            \label{eq:p<min...}
            \|p\| < \min\{1,\nu/\delta '\}.
        \end{equation}
        In particular, \eqref{eq:ps-p<1} holds. Therefore, Lemma~\ref{lem:plr-to-sharp} gives that
        $$
            0 \in \plr _{c',\delta'} f_1,
        $$
        where $f_1(x)=\alpha f(x)+h(x)$, $h(x)=-p(x)+4\|x\|$,  and $c'=6c$, see \eqref{eq:ps-cdel-prim}. Moreover, $\inf f_1(\delta'B_E^\circ) \ge 0$ and
        \begin{equation}
            \label{eq:0-sharp}
            |\nabla f_1| (x) \ge 1,\quad \forall x\in\delta' B_E^\circ \setminus\{0\}.
        \end{equation}

        Define
        $$
            g_1(x) := \alpha g(x) +h(x).
        $$
        Using \eqref{eq:x-bar-in-plr-fg-and-dom}, \eqref{eq:subd-incl} and Lemma~\ref{lem:plr-to-sharp}, we get
        \begin{equation}
            \label{eq:0-plr-g}
            0 \in \plr_{c',\delta'} g_1.
        \end{equation}
      Since $\pa (\alpha f)=\alpha \pa f$, and $\pa (\alpha g)=\alpha \pa g$ from \eqref{eq:subd-incl} it holds that $\partial (\alpha f) \subset   \partial (\alpha g)$. Now the Sum Theorem -- Theorem~\ref{thm:sum} -- gives that
        $$
            \partial f_1 (x) =\pa \alpha f(x)+\pa h(x)\subset \pa \alpha g(x)+\pa h(x) =\partial g_1(x),\quad\forall x\in\delta B_E^\circ.
        $$
        From Proposition~\ref{pro:subdf-to-slope} it follows that
        \begin{equation}
            \label{eq:slope-fg-1}
            |\nabla f_1|(x) \ge |\nabla g_1|(x),\quad\forall x\in\delta B_E^\circ.
        \end{equation}
        Set
        $$
            Y := \{x\in\delta 'B_E^\circ: f_1(x) \le f_1(x_0)\}.
        $$
        Since $0$ is the minimum of $f_1$ on $\delta 'B_E^\circ$, we have that
        \begin{equation}
            \label{eq:0-in-Y}
            0\in Y.
        \end{equation}
        We claim that
        $$
            Y \subset (\nu + \|x_0\|) B_E.
        $$
        Indeed, if $x$ is such that
        $$
            \nu + \|x_0\| < \|x\| < \delta',
        $$
        then, see \eqref{eq:alpha-choice} and \eqref{eq:p<min...},
        \begin{eqnarray*}
        f_1(x) &=& \alpha f(x)-p(x)+4\|x\| \\
         &>&-\nu -\frac {\nu}{\delta '}\delta '+4(\nu +\| x_0\| \\
         &=&  2\nu + 4\|x_0\|,
         \end{eqnarray*}
          while $f_1(x_0) < \nu +\frac {\nu}{\delta '}\delta '  + 4\|x_0\|=2\nu +4\|x_0\|$, so $x\not\in Y$.

        This means that $Y$ is closed, because $f_1$ is lower semicontinuous. So,  with respect to the canonical metric induced by $\|\cdot\|$, $Y$ is a complete metric space. In addition, Lemma~\ref{lem:nabla-restrict} gives that for
        $$
            \ol f_1 := f_1 \upharpoonright Y,
        $$
        we have
        \begin{equation}
            \label{eq:nablaf2=f1}
            |\nabla \ol f_1| (x) = |\nabla f_1 |(x), \quad\forall x\in Y.
        \end{equation}
        Set
        $$
            \ol g_1 := g_1 \upharpoonright Y.
        $$
        Obviously, $|\nabla \ol g_1|(x) \le |\nabla g_1|(x)$, for all $x\in Y$, so \eqref{eq:slope-fg-1} gives
        $$
            |\nabla \ol f_1|(x) \ge |\nabla \ol g_1|(x),\quad \forall x\in Y.
        $$
        From this and \eqref{eq:0-plr-g} we have that for all $x\in\dom |\nabla \ol f_1|\subset \dom |\nabla \ol g_1|$, and all $y\in Y$ we have
      \begin{eqnarray*}
            \ol g_1(y) &\ge& \ol g_1(x) - |\nabla \ol g_1|(x)\|y-x\| - c(|\nabla \ol g_1|(x)+1)\|y-x\|^2\\
    &\ge& \ol g_1(x) - |\nabla \ol f_1|(x)\|y-x\| - c(|\nabla \ol f_1|(x)+1)\|y-x\|^2
    \end{eqnarray*}
        Finally, since $0$ is the minimum of $\ol f_1$ on $Y$, it is sharp by \eqref{eq:0-sharp} and \eqref{eq:nablaf2=f1}, and we can apply Theorem~\ref{thm:slope-main} to conclude that $\ol f_1\ge \ol g_1$ on $Y$. Since $x_0\in Y$, this means $\alpha f(x_0)+p(x_0)+4\|x_0\| \ge \alpha g(x_0)+p(x_0)+4\|x_0\|$, therefore $f(x_0) \ge g(x_0)$ and we are done.
    \end{proof}

    \begin{proof}[\emph{\textbf{Proof of Theorem~\ref{thm:main}}}]
        We again assume that $\bar x = 0$.

        From \eqref{eq:def-hat-del} and \eqref{eq:ps-cdel-prim} it is clear that $\hat \delta = \delta'/2$.

        For any fixed $x\in \hat\delta B^\circ_E$ such that $|\nabla f|(x)<\infty$ (and there are such $x$'s because of the graphical density of  $\dom|\nabla f|$   in $\dom f$, see e.g. \cite[Lemma~3.7.]{iz-slopes}), Lemma~\ref{lem:determ} applied to $f$ and $g$ and then to $g$ and $f$ on $B^\circ(x;\hat\delta)$ shows that:
        $$
            f(y) = g(y) + (f(x)-g(x)),\quad \forall y\in B^\circ(x;\hat\delta).
        $$
        But $0\in B^\circ(x;\hat\delta)$, so
        $$
            f(0) = g(0) +  (f(x)-g(x)) \iff f(x)-g(x) = f(0) - g(0).
        $$
        That is, for $a := f(0) - g(0)$ we have
        $$
            f(y) = g(y) + a,\quad\forall  y\in B^\circ(x;\hat\delta),\ \forall x \in \hat\delta B^\circ_E \cap \dom |\nabla f|.
        $$
        But because $\dom|\nabla f|$ is dense in $\dom f$ and $0\in\dom f$
        $$
            \hat\delta B^\circ_E \subset \left(\hat\delta B^\circ_E \cap \dom |\nabla f|\right) + \hat\delta B^\circ_E.
        $$
 The proof is then complete.   \end{proof}
\end{lemm}

\bigskip

\textbf{Acknowledgements.} The authors dedicate this article to Prof. Lionel Thibault and express their sincere gratitude for his inspiration and constant support.

\bigskip

\textbf{Funding.} The work of M. Ivanov and N. Zlateva is supported by the European Union-NextGenerationEU, through the National Recovery and Resilience Plan of the Republic of Bulgaria, project SUMMIT BG-RRP-2.004-0008-C01. 

The work of M. Konstantinov is supported by the Bulgarian Ministry of Education and Science under the National
Research Programme Young scientists and postdoctoral students approved by DCM \#206/07.04.2022.

\end{document}